\documentclass[11pt, letterpaper]{article}

\usepackage[utf8]{inputenc} % Safe input encoding
\usepackage[T1]{fontenc}    % Proper font encoding for symbols
\usepackage{amsmath, amssymb, amsthm, amsfonts} % Core math packages
\usepackage{mathtools}      % Extensions for amsmath
\usepackage{microtype}      % Improves typographical spacing
\usepackage{graphicx}       % Mandatory for figures/images
\usepackage{cite}           % Better citation handling
\usepackage{algorithm}
\usepackage{algpseudocode}
\usepackage{aliascnt}
\usepackage{subcaption} % For subfigure environment
\usepackage[margin=1in]{geometry} 

\usepackage[hidelinks, colorlinks=true, linkcolor=blue, citecolor=blue, urlcolor=blue]{hyperref}
\usepackage[nameinlink,capitalize,noabbrev]{cleveref}
\newtheorem{theorem}{Theorem}[section]
\newtheorem{lemma}[theorem]{Lemma}

\theoremstyle{definition}
\newtheorem{definition}[theorem]{Definition}

\theoremstyle{remark}
\newtheorem{remark}{Remark}
\newaliascnt{assumption}{theorem}
\newtheorem{assumption}[assumption]{Assumption}
\aliascntresetthe{assumption}

\newcommand{\R}{\mathbb{R}}
\newcommand{\ip}[2]{\left\langle #1,#2\right\rangle}
\newcommand{\norm}[1]{\left\|#1\right\|}
\newcommand{\proj}{P}
\newcommand{\Gap}{\operatorname{Gap}}
\newcommand{\dist}{\operatorname{dist}}
\crefname{theorem}{Theorem}{Theorems}
\Crefname{theorem}{Theorem}{Theorems}
\crefname{lemma}{Lemma}{Lemmas}
\Crefname{lemma}{Lemma}{Lemmas}
\crefname{proposition}{Proposition}{Propositions}
\Crefname{proposition}{Proposition}{Propositions}
\crefname{corollary}{Corollary}{Corollaries}
\Crefname{corollary}{Corollary}{Corollaries}
\crefname{definition}{Definition}{Definitions}
\Crefname{definition}{Definition}{Definitions}
\crefname{assumption}{Assumption}{Assumptions}
\Crefname{assumption}{Assumption}{Assumptions}
\crefname{remark}{Remark}{Remarks}
\Crefname{remark}{Remark}{Remarks}
\crefname{algorithm}{Algorithm}{Algorithms}
\Crefname{algorithm}{Algorithm}{Algorithms}
\title{\textbf{Residual Feedback for Transformed-State Equilibrium Seeking via General Variational Inequalities}%\thanks{This work is supported by the U.S. National Science Foundation under CAREER Award No. ECCS-2439971.}
}

\author{
Griffin Smith%
\thanks{G. Smith is with Applied Mathematics GIDP, The University of Arizona, Tucson, AZ 85721, USA. Email: \texttt{griffinlsmith@arizona.edu}.}
\and
Afrooz Jalilzadeh%
\thanks{A. Jalilzadeh is with the Department of Systems and Industrial Engineering, The University of Arizona, Tucson, AZ 85721, USA. Email: \texttt{afrooz@arizona.edu}.}
}

\date{} % Or hardcode a specific date to prevent updates on re-compilation

\begin{document}

\maketitle

\begin{abstract}
Motivated by networked systems in which equilibrium conditions apply to a regulated state rather than directly to the decision variable, we study transformed-state general variational inequalities. Using a projection-residual reformulation, we develop residual feedback methods that operate in the decision space without inverting the state mapping $H$. For a one-step method, we establish an $\mathcal{O}(T^{-1/2})$ best-iterate residual rate under solution-restricted cocoercivity and linear convergence under solution-restricted strong monotonicity and residual Lipschitz continuity. We also obtain an $\mathcal{O}(T^{-1/2})$ best-iterate residual rate for a predictor-corrector method applied to the induced GVI residual under monotonicity and Lipschitz continuity. The assumptions are imposed directly on the computable residual and are illustrated by affine network and rank-deficient examples. Finally, we extend the residual framework to generalized quasi-variational inequalities with decision-dependent feasible sets.
\end{abstract}

\section{Introduction and Related Work}
\label{sec:introduction}
Control and coordination problems in modern engineered systems often involve equilibrium conditions coupled through physical or operational constraints. Examples arise in transportation and flow networks \cite{Dafermos1980Traffic,Nagurney1999NetworkEconomics}, as well as in multi-agent systems in which local decisions interact through shared constraints \cite{YiPavel2019,bianchi2022fast}. Variational inequalities (VIs) provide a unified mathematical framework for such equilibrium problems in optimization, networks, economics, and control \cite{FacchineiPang2003}. In many control and coordination settings, however, the variable adjusted by a controller, coordinator, or decision maker need not coincide with the state on which feasibility and equilibrium are imposed. Instead, a decision $x$ may generate a regulated state, output, or response $H(x)$, with the admissibility and equilibrium conditions imposed on this transformed state. A representative example is toll design, in which a chosen
toll induces an equilibrium traffic flow
\cite{ChenGupta2022Toll}.
This structure is captured by a \emph{general variational inequality} (GVI): find $x^*$ such that $H(x^*)\in K$ and
\[
    \langle y-H(x^*),Q(x^*)\rangle\geq 0,
    \qquad \forall y\in K.
\]
GVIs of this form, up to notation, have been studied using implicit methods \cite{He1999GVI} and projection-based schemes \cite{NoorWangXiu2002}; projection-type error bounds for this problem class were developed in \cite{XiuZhang2002GVI}. A closely related specialization is the inverse variational inequality (IVI), obtained when the operator in the equilibrium inequality is the original decision variable. Projection-based methods for IVIs were developed in \cite{HeLiu2011}, while more recent work has considered stochastic IVIs with convergence-rate guarantees \cite{AlizadehPolancoJalilzadeh2023} and regularized projection methods for monotone IVIs \cite{SmithAlizadehJalilzadeh2026}. Unlike \cite{SmithAlizadehJalilzadeh2026}, we consider a general pair $(H,Q)$ and unregularized decision-space iterations under conditions imposed on the induced residual. 

Through its projection residual, the GVI becomes a root-finding
problem in decision space, enabling forward and extragradient
iterations under cocoercivity-type or monotone-Lipschitz conditions
\cite{gorbunov2022extragradient,korpelevich1976extragradient};
related schemes appear in stochastic generalized Nash equilibria
\cite{franci2025variance}. Unlike the implicit proximal GVI
framework in \cite{LiLiaoYuan2009}, our method uses explicit
residual steps. It avoids inversion of $H$ and derives verifiable
conditions on $(H,Q)$ for affine network models.

Our decision-dependent extension is related to quasi-variational inequalities (QVIs), which model state-dependent constraints \cite{BensoussanLions1984} and generalized Nash games with coupled feasible sets \cite{Harker1991,AlizadehJalilzadeh2026}. Chan and Pang \cite{ChanPang1982} introduced GQVIs and established existence results. Here, we extend our residual framework to transformed-state GQVIs with decision-dependent feasible sets.

\subsection{Contribution}
We make four contributions. (i) Rather than transforming the problem through $H^{-1}$, we formulate the GVI as a residual equation and develop a one-projection update directly in the decision variable. We establish an $O(T^{-1/2})$ best-iterate residual rate under solution-restricted cocoercivity and linear convergence under solution-restricted strong monotonicity and residual Lipschitz continuity, without requiring invertibility of $H$ or global monotonicity of $H$ and $Q$. (ii) For monotone Lipschitz residuals that need not be cocoercive, we apply a predictor-corrector step and obtain an $O(T^{-1/2})$ best-iterate rate. A rank-deficient example shows that this correction can converge when the one-step method diverges. (iii) For affine network models, we derive directly verifiable spectral conditions on the original matrices that guarantee residual monotonicity, strong monotonicity, and cocoercivity. (iv) Finally, we extend the framework to generalized QVIs with $H(x)\in K(x)$ and give sufficient conditions for Lipschitz continuity of the resulting moving-set residual.
\section{Transformed-State GVI and Projection Residual}
\label{sec:problem-residual}
Let $K\subseteq \R^n$ be nonempty, closed, and convex, and let $H,Q:\R^n\to\R^n$ be continuous. We study the following general variational inequality.
\begin{definition}[General variational inequality]
\label{def:gvi}
The general variational inequality problem, denoted by $\operatorname{GVI}(H,Q,K)$, is to find $x^*\in\R^n$ such that $H(x^*)\in K$ and
\begin{equation}
\ip{y-H(x^*)}{Q(x^*)}\geq 0,
\qquad \forall y\in K.
\label{eq:gvi}
\end{equation}
\end{definition}
The distinction between $x^*$ and $H(x^*)$ is central. In a classical VI, the point satisfying the variational inequality is the decision variable itself. In \Cref{def:gvi}, the decision variable is $x^*$, while the constrained equilibrium point is $H(x^*)$. Hence the model can represent equilibrium conditions involving transformed variables, implicit responses, or state-output maps. If $H(x)=x$, the problem reduces to the classical VI. If, in addition, $Q=\nabla f$ for a differentiable convex function $f$, it becomes the first-order optimality condition for $\min_{x\in K}f(x)$. If $Q(x)=x$, the model becomes an inverse variational inequality. If $H$ is bijective and $F(u)=Q(H^{-1}(u))$, the problem can be transformed into a classical VI in $u=H(x)$. This transformation is often undesirable in network and control settings because $H$ may be noninvertible, expensive to invert, nonsmooth, or available only through an oracle. We focus on the same-dimensional setting in which the transformed residual can be used directly as a decision-space feedback signal. In this interpretation, $x$ is a controller or coordinator signal, $H(x)$ is the resulting network response, and the residual provides an equilibrium-regulation error used to update the decision variable.
The following standard projection characterization \cite{NoorWangXiu2002} motivates our residual feedback law.
\begin{lemma}[Projection characterization]
\label{lem:projection-characterization}
Let $\eta>0$. A point $x^*\in\R^n$ solves $\operatorname{GVI}(H,Q,K)$ if and only if
\begin{equation}
H(x^*)=\proj_K\big(H(x^*)-\eta Q(x^*)\big).
\label{eq:projection-form}
\end{equation}
\end{lemma}
\begin{proof}
Recall that $u=\proj_K(v)$ if and only if $u\in K$ and $\ip{y-u}{v-u}\leq 0$ for all $y\in K$. Let $u=H(x^*)$ and $v=H(x^*)-\eta Q(x^*)$. Then $u=\proj_K(v)$ is equivalent to $H(x^*)\in K$ and $\ip{y-H(x^*)}{-\eta Q(x^*)}\leq 0$ for all $y\in K$. Since $\eta>0$, this is equivalent to \eqref{eq:gvi}.
\end{proof}
Following the standard projection residual for GVIs \cite{NoorWangXiu2002,XiuZhang2002GVI}, we use the following residual.
\begin{definition}[Projection residual]
\label{def:residual}
For $\eta>0$, define $R_\eta:\R^n\to\R^n$ by
\begin{equation}
R_\eta(x):=H(x)-\proj_K\big(H(x)-\eta Q(x)\big).
\label{eq:residual}
\end{equation}
The associated normalized residual measure is $\Gap_\eta(x):=\eta^{-1}\norm{R_\eta(x)}$.
\end{definition}
By \Cref{lem:projection-characterization}, $x^*$ solves $\operatorname{GVI}(H,Q,K)$ if and only if $R_\eta(x^*)=0$. Thus, the GVI can be studied as a residual root-finding problem in the original variable $x$.
\section{Residual Feedback Algorithm and Assumptions}
\label{sec:algorithm-assumptions}
\begin{assumption}[Feasible set and solutions]
\label{ass:K-solution}
The set $K\subseteq\R^n$ is nonempty, closed, and convex. The solution set $X^*:=\{x\in\R^n:R_\eta(x)=0\}$ is nonempty and closed. For each $x\in\R^n$, $P_{X^*}(x)$ denotes the set of Euclidean projections of $x$ onto $X^*$.
\end{assumption}
\begin{assumption}[Residual Lipschitz continuity]
\label{ass:residual-lipschitz}
There exists $L_R>0$ such that
$\norm{R_\eta(x)-R_\eta(y)}\leq L_R\norm{x-y}$ for all $x,y\in\R^n$.
\end{assumption}
\begin{lemma}[A sufficient condition for residual Lipschitz continuity]
\label{lem:residual-lipschitz}
Suppose \Cref{ass:K-solution} holds and $H$ and $Q$ are Lipschitz continuous with constants $L_H>0$ and $L_Q>0$, respectively. Then \Cref{ass:residual-lipschitz} holds with
$L_R=2L_H+\eta L_Q$.
\end{lemma}
\begin{proof}
The result follows from the nonexpansiveness of $\proj_K$:
\begin{align}\nonumber
\norm{R_\eta(x)-R_\eta(y)} &\leq \norm{H(x)-H(y)} +\norm{\proj_K(H(x)-\eta Q(x))-\proj_K(H(y)-\eta Q(y))} \\
&\leq 2\norm{H(x)-H(y)}+\eta\norm{Q(x)-Q(y)}.
\end{align}
Lipschitz continuity of $H$ and $Q$ gives the desired result.
\end{proof}
\begin{assumption}[Solution-restricted cocoercivity]
\label{ass:restricted-cocoercive}
There exists $\beta>0$ such that, for every $x\in\R^n$ and every $\bar x\in P_{X^*}(x)$,
$\ip{R_\eta(x)}{x-\bar x}\geq \beta\norm{R_\eta(x)}^2.$
\end{assumption}
\begin{assumption}[Solution-restricted strong monotonicity]
\label{ass:restricted-strong}
There exists $\mu>0$ such that, for every $x\in\R^n$ and every $\bar x\in P_{X^*}(x)$,
$\ip{R_\eta(x)}{x-\bar x}\geq \mu\dist(x,X^*)^2.$
\end{assumption}
For a unique solution, these assumptions reduce to cocoercivity
around the solution and quasi-strong monotonicity
\cite{LoizouEtAl2021}; here they extend to nonunique solution
sets by testing against nearest solutions, without requiring
global monotonicity of the residual.
\begin{remark}[Recovery of VI and IVI conditions]
The residual assumptions recover standard VI and inverse-VI cases. If $H(x)=x$ and $Q$ is $\beta_Q$-cocoercive, then, for $0<\eta<4\beta_Q$, $R_\eta$ is $(1-\eta/(4\beta_Q))$-cocoercive. If $H(x)=x$ and $Q$ is $\mu_Q$-strongly monotone and $L_Q$-Lipschitz, then, for $0<\eta<2\mu_Q/L_Q^2$, $R_\eta$ is strongly monotone with modulus
$1-\sqrt{1-2\eta\mu_Q+\eta^2L_Q^2}.$
For the inverse-VI case $Q(x)=x$, if $H$ is $\mu_H$-strongly monotone and $L_H$-Lipschitz, firm nonexpansiveness of $P_K$ gives
$$\langle R_\eta(x)-R_\eta(y),x-y\rangle
\geq
\left(\mu_H-\frac{L_H^2}{4\eta}\right)\|x-y\|^2.$$
Hence $R_\eta$ is strongly monotone whenever
$\eta>L_H^2/(4\mu_H)$, without requiring $H$ to be invertible. Therefore, ~\Cref{ass:restricted-cocoercive} holds in the first case and ~\Cref{ass:restricted-strong} holds in the latter two; residual Lipschitz continuity also makes the latter two sufficient for ~\Cref{ass:restricted-cocoercive}.
\end{remark}

\begin{algorithm}[htp]
\caption{Residual Feedback for Transformed-State GVI}
\label{alg:gvi-residual}
\begin{algorithmic}[1]
\State Choose $x_0\in\R^n$, projection parameter $\eta>0$, and stepsize $\alpha>0$.
\For{$k=0,1,2,\ldots$}
    \State $z_k=\proj_K\big(H(x_k)-\eta Q(x_k)\big)$.
    \State $R_\eta(x_k)=H(x_k)-z_k$.
    \State $x_{k+1}=x_k-\alpha R_\eta(x_k)$.
\EndFor
\end{algorithmic}
\end{algorithm}
\Cref{alg:gvi-residual} uses the residual as a correction direction: $x_{k+1}=x_k-\alpha R_\eta(x_k)$. Each iteration requires one projection onto $K$ and no inversion of $H$, and its fixed points are precisely the solutions of the GVI.
\subsection{A predictor-corrector variant for monotone residuals}
\label{subsec:monotone-eg}
The residual-correction method in \Cref{alg:gvi-residual} requires a cocoercivity-type condition to guarantee descent. When only monotonicity of the residual is available, in \Cref{alg:gvi-eg} we apply the classical extragradient construction~\cite{korpelevich1976extragradient} to the unconstrained residual equation $R_\eta(x)=0$. Modern analyses of extragradient methods establish residual guarantees for general monotone Lipschitz operators~\cite{gorbunov2022extragradient}, while  \cite{franci2025variance} considers extragradient variants for stochastic equilibrium problems. Here, the operator is not a primitive VI mapping but the projection residual induced by the transformed-state GVI.
\begin{assumption}[Monotonicity of the residual]
\label{ass:residual-monotone}
The residual mapping $R_\eta$ is monotone; that is, for all $x,y\in\R^n$,
$\ip{R_\eta(x)-R_\eta(y)}{x-y}\geq 0.$
\end{assumption}
\begin{algorithm}[htp]
\caption{Predictor-Corrector Residual Feedback Method}
\label{alg:gvi-eg}
\begin{algorithmic}[1]
\State Choose $x_0\in\R^n$, projection parameter $\eta>0$, and stepsize $\alpha>0$.
\For{$k=0,1,2,\ldots$}
    \State Compute the predictor $\widehat x_k=x_k-\alpha R_\eta(x_k)$.
    \State Compute the corrected update $x_{k+1}=x_k-\alpha R_\eta(\widehat x_k)$.
\EndFor
\end{algorithmic}
\end{algorithm}
\section{Convergence Analysis}
\label{sec:convergence}
We now establish convergence guarantees for the two residual feedback methods under the residual-level assumptions introduced in \Cref{sec:algorithm-assumptions}.
\begin{theorem}[Sublinear residual convergence]
\label{thm:sublinear}
Suppose \Cref{ass:K-solution,ass:restricted-cocoercive} hold. Let $\{x_k\}$ be generated by \Cref{alg:gvi-residual} with constant stepsize $\alpha\in(0,2\beta)$. Then, for every $k\geq 0$,
\begin{equation}
\dist(x_{k+1},X^*)^2
\leq
\dist(x_k,X^*)^2-\alpha(2\beta-\alpha)\norm{R_\eta(x_k)}^2.
\label{eq:fejer}
\end{equation}
Consequently, for every $T\geq 1$,
\begin{equation}
\min_{0\leq k\leq T-1}\Gap_\eta(x_k)^2
\leq
\frac{\dist(x_0,X^*)^2}{\eta^2\alpha(2\beta-\alpha)T}.
\label{eq:sublinear-rate}
\end{equation}
Thus, $\min_{0\leq k\leq T-1}\Gap_\eta(x_k)=\mathcal{O}(T^{-1/2})$ and hence $\norm{R_\eta(x_k)}\to0$
\end{theorem}
\begin{proof}
Let $\bar x_k\in P_{X^*}(x_k)$. Since $\bar x_k\in X^*$, we have
$R_\eta(\bar x_k)=0$. Moreover, by the definition of the projection onto
$X^*$, $\norm{x_k-\bar x_k}=\dist(x_k,X^*)$. Using the update
$x_{k+1}=x_k-\alpha R_\eta(x_k)$, we obtain
\begin{align}
\dist(x_{k+1},X^*)^2
&\leq \norm{x_{k+1}-\bar x_k}^2 = \norm{x_k-\alpha R_\eta(x_k)-\bar x_k}^2 \notag\\
&= \norm{x_k-\bar x_k}^2
-2\alpha\ip{R_\eta(x_k)}{x_k-\bar x_k}
+\alpha^2\norm{R_\eta(x_k)}^2 \notag\\
&= \dist(x_k,X^*)^2
-2\alpha\ip{R_\eta(x_k)}{x_k-\bar x_k}
+\alpha^2\norm{R_\eta(x_k)}^2.
\label{eq:basic-expansion}
\end{align}
By \Cref{ass:restricted-cocoercive}, we have
\begin{equation}
\ip{R_\eta(x_k)}{x_k-\bar x_k}
\geq
\beta\norm{R_\eta(x_k)}^2.
\end{equation}
Substituting this bound into \eqref{eq:basic-expansion} gives
\begin{align}
\dist(x_{k+1},X^*)^2
&\leq
\dist(x_k,X^*)^2
-2\alpha\beta\norm{R_\eta(x_k)}^2 +\alpha^2\norm{R_\eta(x_k)}^2 \notag\\
&=
\dist(x_k,X^*)^2
-\alpha(2\beta-\alpha)\norm{R_\eta(x_k)}^2.
\end{align}
This proves \eqref{eq:fejer}. Since $\alpha\in(0,2\beta)$, the coefficient
$\alpha(2\beta-\alpha)$ is positive.
Now summing \eqref{eq:fejer} from $k=0$ to $T-1$ yields
\begin{align}
\alpha(2\beta-\alpha)\sum_{k=0}^{T-1}\norm{R_\eta(x_k)}^2
&\leq
\sum_{k=0}^{T-1}
\left(\dist(x_k,X^*)^2-\dist(x_{k+1},X^*)^2\right) \notag \\
&=
\dist(x_0,X^*)^2-\dist(x_T,X^*)^2
\leq
\dist(x_0,X^*)^2.
\label{eq:sum-residual-bound}
\end{align}
Therefore,
\begin{equation}
\sum_{k=0}^{T-1}\norm{R_\eta(x_k)}^2
\leq
\frac{\dist(x_0,X^*)^2}{\alpha(2\beta-\alpha)}.
\end{equation}
Letting $T\to\infty$ gives $\sum_{k=0}^{\infty}\norm{R_\eta(x_k)}^2<\infty$, and hence $\norm{R_\eta(x_k)}\to0$.
Since the minimum is bounded by the average, we have
\begin{align}
\min_{0\leq k\leq T-1}\norm{R_\eta(x_k)}^2
&\leq
\frac{1}{T}\sum_{k=0}^{T-1}\norm{R_\eta(x_k)}^2 
\leq
\frac{\dist(x_0,X^*)^2}{\alpha(2\beta-\alpha)T}.
\end{align}
Finally, using $\Gap_\eta(x)=\eta^{-1}\norm{R_\eta(x)}$, we obtain
\begin{align}
\min_{0\leq k\leq T-1}\Gap_\eta(x_k)^2
&=
\frac{1}{\eta^2}
\min_{0\leq k\leq T-1}\norm{R_\eta(x_k)}^2
\leq
\frac{\dist(x_0,X^*)^2}
{\eta^2\alpha(2\beta-\alpha)T}.
\end{align}
This proves \eqref{eq:sublinear-rate}. Taking square roots gives
$\min_{0\leq k\leq T-1}\Gap_\eta(x_k)=\mathcal{O}(T^{-1/2})$.
\end{proof}
The preceding result shows that solution-restricted cocoercivity is sufficient to guarantee a sublinear residual rate for the one-step residual feedback method. We next show that if the residual additionally satisfies solution-restricted strong monotonicity and is Lipschitz continuous, then the same iteration converges linearly in distance to the solution set.
\begin{theorem}[Linear convergence]
\label{thm:linear}
Suppose \Cref{ass:K-solution,ass:residual-lipschitz,ass:restricted-strong} hold. Let $L_R$ denote the Lipschitz constant of $R_\eta$ in \Cref{ass:residual-lipschitz}; for example, \Cref{lem:residual-lipschitz} gives $L_R=2L_H+\eta L_Q$ when $H$ and $Q$ are Lipschitz continuous. Let $\{x_k\}$ be generated by \Cref{alg:gvi-residual} with constant stepsize $\alpha\in(0,2\mu/L_R^2)$. Then, for every $k\geq 0$,
\begin{align}\nonumber
&\dist(x_{k+1},X^*)^2
\leq
\rho\dist(x_k,X^*)^2,\\
&\rho:=1-2\alpha\mu+\alpha^2L_R^2.
\label{eq:linear-rate}
\end{align}
Moreover, $\rho<1$; if $X^*\neq\R^n$, then $\rho\in[0,1)$. In particular, choosing $\alpha=\mu/L_R^2$ gives $\dist(x_{k+1},X^*)^2\leq (1-\mu^2/L_R^2)\dist(x_k,X^*)^2$.
\end{theorem}
\begin{proof}
Let $\bar x_k\in P_{X^*}(x_k)$. Using the update and the same expansion as in \eqref{eq:basic-expansion},
\begin{align}
\dist(x_{k+1},X^*)^2
&\leq \dist(x_k,X^*)^2
-2\alpha\ip{R_\eta(x_k)}{x_k-\bar x_k} +\alpha^2\norm{R_\eta(x_k)}^2.
\end{align}
By \Cref{ass:restricted-strong} we bound the second term:
\[
-2\alpha\ip{R_\eta(x_k)}{x_k-\bar x_k} \leq -2\alpha \mu \dist(x_k,X^*)^2.
\]
By \Cref{ass:residual-lipschitz} we bound the last term using the Lipschitz constant $L_R$:
\begin{align*}
\norm{R_\eta(x_k) - R_\eta(\bar x_k)} \leq L_R \norm{x_k - \bar x_k}.
\end{align*}
As $\bar x_k$ is an element of $X^*$ we have $R_\eta(\bar x_k)=0$:
\begin{align*}
\norm{R_\eta(x_k)} \leq L_R \norm{x_k - \bar x_k}.
\end{align*}
Thus we get
\begin{align*}
\dist(x_{k+1},X^*)^2
&\leq \dist(x_k,X^*)^2
-2\alpha\mu\dist(x_k,X^*)^2
+ \alpha^2 L_R^2\dist(x_k,X^*)^2\\
&\leq (1 - 2\alpha\mu + \alpha^2L_R^2)\dist(x_k,X^*)^2.
\end{align*}
Letting $\rho = 1 - 2\alpha\mu + \alpha^2L_R^2$ yields the desired result. Since $\alpha\in(0,2\mu/L_R^2)$, we have $\rho<1$.
If $X^*\neq\R^n$, choose $x\notin X^*$ and $\bar x\in P_{X^*}(x)$. By \Cref{ass:restricted-strong}, the Cauchy--Schwarz inequality, and the Lipschitz continuity of $R_\eta$, we have
$
\mu \operatorname{dist}(x,X^*)^2
\leq
\langle R_\eta(x),x-\bar{x}\rangle
\leq
\norm{R_\eta(x)}\operatorname{dist}(x,X^*)
\leq
L_R\operatorname{dist}(x,X^*)^2.
$
Since $\operatorname{dist}(x,X^*)>0$, it follows that $\mu\leq L_R$. The quadratic $1-2\alpha\mu+\alpha^2L_R^2$ is minimized at $\alpha=\mu/L_R^2$, and therefore
$\rho\geq1-\mu^2/L_R^2\geq0$.
Hence $\rho\in[0,1)$. Choosing $\alpha=\mu/L_R^2$ gives $\rho=1-\mu^2/L_R^2$.
\end{proof}
The next theorem shows that the additional correction step in \Cref{alg:gvi-eg} compensates for the lack of cocoercivity. Monotonicity and Lipschitz continuity of the residual mapping are sufficient to obtain a nonasymptotic residual bound.
\begin{theorem}[Residual convergence under monotonicity]
\label{thm:monotone-eg}
Suppose \Cref{ass:K-solution,ass:residual-lipschitz,ass:residual-monotone} hold. Let $\{x_k\}$ be generated by \Cref{alg:gvi-eg} with constant stepsize $\alpha\in(0,1/L_R)$. Then, for every $k\geq 0$,
\begin{equation}
\dist(x_{k+1},X^*)^2
\leq
\dist(x_k,X^*)^2
-\alpha^2(1-\alpha^2L_R^2)\norm{R_\eta(x_k)}^2.
\label{eq:eg-descent}
\end{equation}
Consequently, for every $T\geq 1$,
\begin{equation}
\min_{0\leq k\leq T-1}\Gap_\eta(x_k)^2
\leq
\frac{\dist(x_0,X^*)^2}
{\eta^2\alpha^2(1-\alpha^2L_R^2)T}.
\label{eq:eg-rate}
\end{equation}
Thus, $\min_{0\leq k\leq T-1}\Gap_\eta(x_k)=\mathcal{O}(T^{-1/2})$. Moreover, $\norm{R_\eta(x_k)}\to0$.
\end{theorem}
\begin{proof}
Let $\bar x_k\in P_{X^*}(x_k)$. Since $\bar x_k\in X^*$, we have $R_\eta(\bar x_k)=0$. The update of \Cref{alg:gvi-eg} gives
\begin{align}
\dist(x_{k+1},X^*)^2
&\leq
\norm{x_{k+1}-\bar x_k}^2
=
\norm{x_k-\bar x_k-\alpha R_\eta(\widehat x_k)}^2 \notag \\
&=
\dist(x_k,X^*)^2
-2\alpha\ip{R_\eta(\widehat x_k)}{x_k-\bar x_k}
+
\alpha^2\norm{R_\eta(\widehat x_k)}^2 .
\label{eq:eg-proof-1}
\end{align}
Consider the cross term $\ip{R_\eta(\hat x_k)}{x_k - \bar x_k}$ and rewrite
\[x_k - \bar x_k = x_k - \hat x_k + \hat x_k - \bar x_k = \alpha R_\eta(x_k) + (\hat x_k - \bar x_k).\]
Then
\[
\ip{R_\eta(\hat x_k)}{x_k - \bar x_k} = \alpha \ip{R_\eta(\hat x_k)}{R_\eta(x_k)} + \ip{R_\eta(\hat x_k)}{\hat x_k - \bar x_k}.
\]
This second term may be bounded from below using \Cref{ass:residual-monotone}.
As $R_\eta(\bar x_k)=0$ we write
\[
\ip{R_\eta(\hat x_k)}{\hat x_k - \bar x_k} = \ip{R_\eta(\hat x_k) - R_\eta(\bar x_k)}{\hat x_k - \bar x_k} \geq 0.
\]
Combining with \Cref{eq:eg-proof-1},
\begin{align}
\nonumber \dist(x_{k+1}, X^*)^2
&\leq \dist(x_k,X^*)^2 - 2\alpha^2\ip{R_\eta(\hat x_k)}{R_\eta(x_k)} + \alpha^2 \norm{R_\eta(\hat x_k)}^2.
\end{align}
Observe that
\[
-2 \ip{R_\eta(x_k)}{R_\eta(\hat x_k)} + \norm{R_\eta(\hat x_k)}^2
= \norm{R_\eta(\hat x_k)-R_\eta(x_k)}^2 - \norm{R_\eta(x_k)}^2.
\]
The norm squared of the residuals is bounded using \Cref{ass:residual-lipschitz} to achieve
\begin{align}
\nonumber \dist(x_{k+1}, X^*)^2
&\leq \dist(x_k,X^*)^2 - \alpha^2(1 - \alpha^2L_R^2) \norm{R_\eta(x_k)}^2.
\end{align}
We proceed in a similar way to \Cref{thm:sublinear}.
Summing from $k=0$ to $k=T-1$,
\begin{align}
\nonumber &\alpha^2(1 - \alpha^2L_R^2) \sum_{k=0}^{T-1} \norm{R_\eta(x_k)}^2
\leq \dist(x_0,X^*)^2 - \dist(x_T,X^*)^2
\leq \dist(x_0,X^*)^2.
\end{align}
Letting $T\to\infty$ gives $\sum_{k=0}^{\infty}\norm{R_\eta(x_k)}^2<\infty$, and hence $\norm{R_\eta(x_k)}\to0$.
As the minimum is bounded by the average,
\begin{align}
\nonumber \min_{0 \le k \le T-1} \|R_\eta(x_k)\|^2
&\le \frac{1}{T} \sum_{k=0}^{T-1} \|R_\eta(x_k)\|^2
\le \frac{\dist(x_0, X^*)^2}{\alpha^2 (1 - \alpha^2 L_R^2) T}.
\end{align}
Finally using $\Gap_\eta(x)=\eta^{-1}\norm{R_\eta(x)}$ we obtain the desired result
\begin{equation*}
\min_{0\leq k\leq T-1}\Gap_\eta(x_k)^2
\leq
\frac{\dist(x_0,X^*)^2}
{\eta^2\alpha^2(1-\alpha^2L_R^2)T}.
\end{equation*}
Taking the square root gives $\min_{0\leq k\leq T-1}\Gap_\eta(x_k)=\mathcal{O}(T^{-\tfrac{1}{2}})$.
\end{proof}
\begin{remark}
\label{rem:monotone-vs-cocoercive}
Unlike the one-step method in \Cref{alg:gvi-residual}, \Cref{alg:gvi-eg} requires monotonicity and Lipschitz continuity of $R_\eta$ rather than solution-restricted cocoercivity, at the cost of one additional residual evaluation per iteration. The numerical example in \Cref{sec:numerics} illustrates a monotone non-cocoercive setting where the one-step method diverges while the predictor-corrector method converges.
\end{remark}
\section{Verifiable Conditions and Network Example}
\label{sec:example}
We now give a network setting in which the residual assumptions can be verified directly. Such models arise naturally in traffic and network-flow equilibrium problems \cite{Dafermos1980Traffic,Nagurney1999NetworkEconomics}.

\textbf{Network-flow equilibrium with directly verifiable residual conditions.}
Let $H(x)=Ax+a$ denote the realized network state and let
$Q(x)=Bx+b$ represent a marginal cost, congestion, price, or equilibrium
residual. Consider the nonempty network-feasibility set
\begin{equation}
K=\{u\in\R^n:Mu=d,\;0\leq u\leq c\},
\label{eq:polyhedral-K}
\end{equation}
where $Mu=d$ imposes flow-conservation constraints and
$0\leq u\leq c$ represents capacity constraints. The residual is
\begin{align}
R_\eta(x)
=
Ax+a-\proj_K\big((A-\eta B)x+a-\eta b\big).
\end{align}
For any matrix $C$, let
$\operatorname{sym}(C):=(C+C^\top)/2$.
By nonexpansiveness of $\proj_K$, for all $x,y\in\R^n$,
\begin{align*}
\|R_\eta(x)-R_\eta(y)\|
&\leq
\bigl(\|A\|+\|A-\eta B\|\bigr)\|x-y\|.
\end{align*}
Thus, $R_\eta$ is Lipschitz continuous with
$L_R:=\|A\|+\|A-\eta B\|.$
To establish monotonicity, let $\Delta x:=x-y$ and
\begin{align*}
\Delta p
:={}&
\proj_K\big((A-\eta B)x+a-\eta b\big)
-
\proj_K\big((A-\eta B)y+a-\eta b\big).
\end{align*}
Firm nonexpansiveness of $\proj_K$ gives
$\|\Delta p\|^2
\leq
\langle \Delta p,(A-\eta B)\Delta x\rangle,$
or, equivalently,
\[
\left\|
\Delta p-\frac12(A-\eta B)\Delta x
\right\|
\leq
\frac12\|(A-\eta B)\Delta x\|.
\]
Hence, using Cauchy–Schwarz we obtain,
\begin{align*}
\langle\Delta p,\Delta x\rangle
&=
\frac12\langle(A-\eta B)\Delta x,\Delta x\rangle
+
\left\langle
\Delta p-\frac12(A-\eta B)\Delta x,\Delta x
\right\rangle\\
&\leq
\frac12\langle(A-\eta B)\Delta x,\Delta x\rangle
+
\frac12\|(A-\eta B)\Delta x\|\,\|\Delta x\|,
\end{align*}
Using
$R_\eta(x)-R_\eta(y)=A\Delta x-\Delta p$, we obtain
\begin{align*}
\langle R_\eta(x)-R_\eta(y),\Delta x\rangle
&\geq
\frac12\langle(A+\eta B)\Delta x,\Delta x\rangle
-\frac12\|(A-\eta B)\Delta x\|\,\|\Delta x\| \\
&\geq
\frac12\lambda_{\min}\!\big(\operatorname{sym}(A+\eta B)\big)
\|\Delta x\|^2-
\frac12\|A-\eta B\|\|\Delta x\|^2\\
&=
\mu_R\|\Delta x\|^2,
\end{align*}
where
$\mu_R
:=
\frac12\left(
\lambda_{\min}\!\big(\operatorname{sym}(A+\eta B)\big)
-\|A-\eta B\|
\right).$
Therefore, if
\begin{equation}
\lambda_{\min}\!\big(\operatorname{sym}(A+\eta B)\big)
\geq
\|A-\eta B\|,
\label{eq:network-monotonicity-condition}
\end{equation}
then $R_\eta$ is monotone. Provided the solution set is nonempty, this and
Lipschitz continuity verify the assumptions of \Cref{thm:monotone-eg}.
If \eqref{eq:network-monotonicity-condition} is strict, then
$\mu_R>0$ and $R_\eta$ is strongly monotone. Under \Cref{ass:K-solution},
strong monotonicity implies that the zero of $R_\eta$ is unique, and
\Cref{ass:restricted-strong} holds. Moreover,
\begin{align*}
\langle R_\eta(x)-R_\eta(y),x-y\rangle
&\geq
\frac{\mu_R}{L_R^2}
\|R_\eta(x)-R_\eta(y)\|^2.
\end{align*}
Thus, \Cref{ass:restricted-cocoercive} holds with
\[
\beta_R
=
\frac{
\lambda_{\min}\!\big(\operatorname{sym}(A+\eta B)\big)
-\|A-\eta B\|
}{
2\bigl(\|A\|+\|A-\eta B\|\bigr)^2
}.
\]
Two useful special cases illustrate the condition. If $B=\gamma A$,
$\gamma>0$, and $A=A^\top\succ0$, then
\eqref{eq:network-monotonicity-condition} reduces to
$(1+\eta\gamma)\lambda_{\min}(A)
\geq
|1-\eta\gamma|\,\|A\|,$
which holds strictly, in particular, when $\eta\gamma=1$. This case arises
when marginal costs are proportional to the realized network state.
More simply, if $A=\nu I$ and $B=\gamma I$ with $\nu,\gamma>0$, then the
condition is strict for every $\eta>0$, with
$\mu_R=\min\{\nu,\eta\gamma\}.$

\section{Numerical Illustration}
\label{sec:numerics}

{\bf Example 1.} We first consider the network model of Section \ref{sec:example} on a connected directed graph with $n_v=5$ nodes and $m=8$ arcs. Its incidence matrix is denoted by $M$, and
\begin{gather*}
u^\star=(1.5,1,0.5,1,1,0.5,1,1)^\intercal, \quad
d=(-3,2,1,2,-2)^\intercal,\\
c=(2,2,2,1,1,1,2,2)^\intercal, \quad
a=u^\star,\quad
b=(-1,-2,-3,-2,-2,2,-1,-3)^\intercal .
\end{gather*}
Let $W$ be the directed line-graph adjacency matrix, where $W_{ij}=1$ when the head of arc $i$ is the tail of arc $j$, and define
\[
G=\frac{W-W^\intercal}{\norm{W-W^\intercal}},
\qquad
A_{ii}=1+\frac{0.2(i-1)}{m-1}.
\]
We set $\eta=1$, $B=A+G$, $H(x)=Ax+a$, and $Q(x)=Bx+b$. Since $\norm{G}=1$, we get
${\mu}_R
=\frac{\lambda_{\min}(\operatorname{sym}(A+B))-\norm{A-B}}{2}
=\frac12,$ and $
 L_R=\norm A+\norm{A-B}=2.2.$
Thus, the residual is strongly monotone and Lipschitz continuous. The stepsize $\alpha=0.1860$ satisfies both
$\alpha<2{\mu}_R/ L_R^2$ and
$\alpha<1/ L_R$. Figure~\ref{fig:smOracle} compares the two methods using residual evaluations, accounting for the additional evaluation required by the predictor-corrector method.
\begin{figure}[t]
    \centering
    \begin{subfigure}[t]{0.4\columnwidth}
        \centering
        \includegraphics[width=\linewidth]{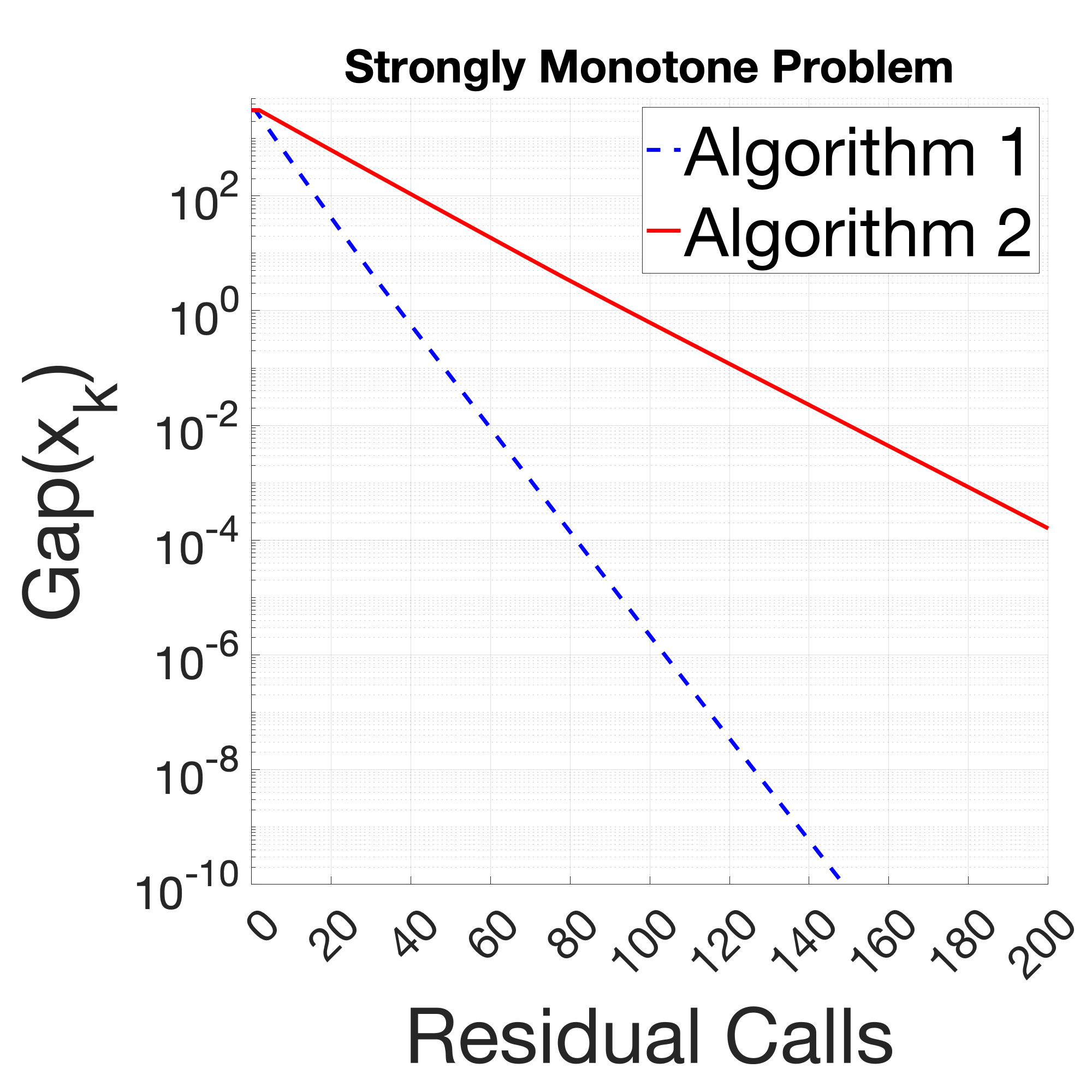}
        \caption{}
        \label{fig:smOracle}
    \end{subfigure}%
    \hfill
    \begin{subfigure}[t]{0.4\columnwidth}
        \centering
        \includegraphics[width=\linewidth]{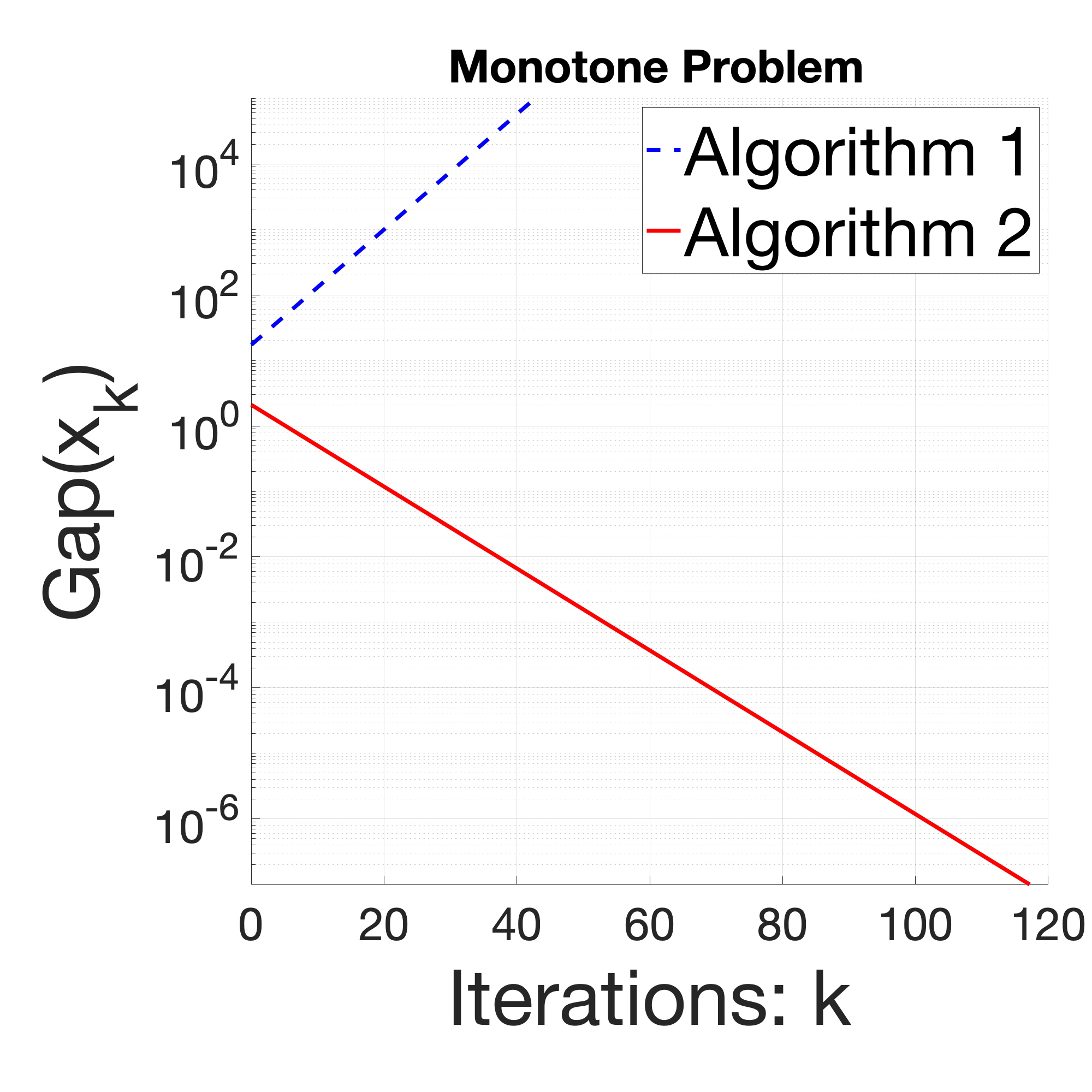}
        \caption{}
        \label{fig:monotone}
    \end{subfigure}

    \caption{%
    (a) Gap versus residual evaluations for the strongly monotone network
    example with $\eta=1$, $\alpha=0.1860$, and
    $x_0=1000\cdot\mathbf{1}$.
    (b) Gap versus iterations for the  monotone example. The one-step method diverges, whereas the predictor-corrector
    method converges for the monotone rotational residual.}
    \label{fig:networkComparisons}
\end{figure}

{\bf Example 2.} We next give a rank-deficient example illustrating the need for the predictor-corrector step. Let $\eta=1$,
\[
K=\{u\in\R^2:u_2=0\},\quad
A=\begin{pmatrix}1&0\\1&0\end{pmatrix},\quad
B=\begin{pmatrix}0&-1\\0&1\end{pmatrix},
\]
and define $H(x)=Ax$ and $Q(x)=Bx$. The matrix $A$ is singular with a kernel described by $(0,t)^\intercal$, and $Q(0,t) = (-t,t)^\intercal$. Hence $Q$ is not constant on the fibers of $H$, and the change of variables $u=H(x)$ does not produce a well-defined operator $Q\circ H^{-1}$.
The unique solution is $x^\star=0$, and the residual is
\[
R_\eta(x)
=H(x)-\proj_K\bigl(H(x)-Q(x)\bigr)
=Jx,\quad
J=\begin{pmatrix}0&-1\\1&0\end{pmatrix}.
\]
Because $J^\intercal=-J$ and $J^\intercal J=I$, the residual is monotone and $1$-Lipschitz, but it is neither strongly monotone nor cocoercive.
%$\ip{R_\eta(x)}{x-x^\star}=0$ and $\norm{R_\eta(x)}^2=\norm{x}^2>0,$$(x\neq0).$
For Algorithm~\ref{alg:gvi-residual},
$\norm{x_{k+1}}^2=(1+\alpha^2)\norm{x_k}^2,$
so the one-step method diverges for every $\alpha>0$ unless $x_0=0$. In contrast, Algorithm~\ref{alg:gvi-eg} satisfies
$\norm{x_{k+1}}^2
=(1-\alpha^2+\alpha^4)\norm{x_k}^2$
and therefore converges for every $0<\alpha<1=L_R^{-1}$. Figure~\ref{fig:monotone} illustrates this distinction for $\alpha=1/\sqrt{2}$ and $x_0 = 10 \cdot \mathbf{1}$.

\section{GQVI Extension and Future Directions}

We conclude by extending the residual framework to decision-dependent feasible sets. Let
$K:\R^n\rightrightarrows\R^n$ have nonempty, closed, and convex values. The associated generalized quasi-variational inequality (GQVI) seeks $x^*$ such that
$H(x^*)\in K(x^*)$ and
$$\ip{y-H(x^*)}{Q(x^*)}\geq 0,\quad \forall y\in K(x^*).$$
Define the moving-set residual
\begin{equation}
R_\eta^q(x)
:=
H(x)-\proj_{K(x)}\bigl(H(x)-\eta Q(x)\bigr).
\end{equation}
Then $x^*$ solves the GQVI if and only if $R_\eta^q(x^*)=0$. Consequently, Algorithms~\ref{alg:gvi-residual} and~\ref{alg:gvi-eg} apply with $R_\eta$ replaced by
$R_\eta^q$, provided the corresponding residual assumptions hold.

A sufficient condition for Lipschitz continuity of this residual is that, for some $L_v,L_K>0$,
\begin{equation}
\norm{\proj_{K(x)}(u)-\proj_{K(y)}(v)}
\leq
L_v\norm{u-v}+L_K\norm{x-y}.
\end{equation}
If $H$ and $Q$ are Lipschitz continuous with constants $L_H$ and $L_Q$, respectively, then
\begin{align*}
\norm{R_\eta^q(x)-R_\eta^q(y)}
&\leq
\norm{H(x)-H(y)}+L_K\norm{x-y}
+L_v\norm{H(x)-H(y)-\eta\bigl(Q(x)-Q(y)\bigr)}
 \\
&\leq
\left(L_H+L_v(L_H+\eta L_Q)+L_K\right)\norm{x-y}.
\end{align*}
Thus, $R_\eta^q$ is Lipschitz continuous with
$L_R^q\leq L_H+L_v(L_H+\eta L_Q)+L_K.$
This condition applies to structured decision-dependent constraints, including moving box constraints and parameterized polyhedral feasible sets, under standard regularity assumptions on the associated projection mapping \cite{BednarczukRutkowski2021}.

Overall, the proposed residual framework provides implementable decision-space methods for transformed-state equilibria without requiring inversion of $H$, while also accommodating decision-dependent feasibility under suitable regularity conditions. Future work will develop verifiable conditions for broader classes of moving feasible sets and investigate stochastic and distributed residual-feedback methods.

% --- Bibliography ---
\bibliographystyle{IEEEtran}
\bibliography{papers}
\end{document}